\documentclass[12pt]{article}

\usepackage[margin=1in]{geometry}
\usepackage{amsthm}
\usepackage{amssymb}
\usepackage{amsmath}
\usepackage{graphicx}
\usepackage{mathrsfs}
\usepackage{url}
\usepackage{color}
\usepackage{framed}
\usepackage[table]{xcolor}
\usepackage[ruled,vlined]{algorithm2e}
\usepackage{mathtools}

\newtheorem{theorem}{Theorem}

\newtheorem{lemma}[theorem]{Lemma}
\newtheorem{proposition}[theorem]{Proposition}

\theoremstyle{definition}
\newtheorem{problem}[theorem]{Problem}

\newtheorem{definition}[theorem]{Definition}

\newcommand{\lb}{(}
\newcommand{\rb}{)}

\newcommand{\cL}{\mathscr{L}}
\newcommand{\bF}{\mathbf{F}}
\newcommand{\ip}[2]{\left<#1,#2\right>}

\newcommand{\absip}[2]{\left\vert\left<#1,#2\right>\right\vert}

\newcommand{\beq}{\begin{equation}}
\newcommand{\eeq}{\end{equation}}

\newcommand{\ejk}[1]{#1}

\title{Testing isomorphism between tuples of subspaces}

\author{Emily J.~King\footnote{Department of Mathematics, Colorado State University, Fort Collins, Colorado, USA} \and Dustin~G.~Mixon\footnote{Department of Mathematics, The Ohio State University, Columbus, Ohio, USA} \footnote{Translational Data Analytics Institute, The Ohio State University, Columbus, Ohio, USA} \and Shayne Waldron\footnote{Department of Mathematics, The University of Auckland, Auckland, New Zealand}}
\date{}

\begin{document}
\maketitle

\begin{abstract}
Given two tuples of subspaces, can you tell whether the tuples are isomorphic?
We develop theory and algorithms to address this fundamental question.
We focus on isomorphisms in which the ambient vector space is acted on by either a unitary group or general linear group.
If isomorphism also allows permutations of the subspaces, then the problem is at least as hard as graph isomorphism.
Otherwise, we provide a variety of polynomial-time algorithms with Matlab implementations to test for isomorphism. \textbf{Keywords}: subspace isomorphism, Grassmannian, Bargmann invariants, $H^\ast$-algebras, quivers, graph isomorphism \textbf{MSC2020}: 14M15, 05C60, 81R05, 16G20
\end{abstract}

\section{Introduction}

The last decade has seen a surge of work to find arrangements of points in real and complex Grassmannian spaces that are well spaced in some sense; for example, one may seek \textbf{optimal codes}, which maximize the minimum distance between points, or \textbf{designs}, which offer an integration rule.
While precursor work in this vein appeared between the '50s and '80s by Rankin~\cite{Rankin:55,Rankin:56}, Grey~\cite{Grey:62}, Seidel~\cite{Seidel:73}, Welch~\cite{Welch:74}, and Levenshtein~\cite{Levenstein:82}, the seminal paper by Conway, Hardin and Sloane~\cite{ConwayHS:96} arrived later in 1996.
The recent resurgence of interest in this problem has been largely stimulated by emerging applications in multiple description coding~\cite{StrohmerH:03}, digital fingerprinting~\cite{MixonQKF:13}, compressed sensing~\cite{BandeiraFMW:13}, and quantum state tomography~\cite{RenesBSC:04}.

There have been several fruitful approaches to studying arrangements of points in the Grassmannian.
First, it is natural to consider highly symmetric arrangements of points.
Such arrangements were extensively studied in~\cite{ValeW:04,ChienW:11,BroomeW:13,Waldron:13,ValeW:16,ChienW:16} in the context of designs, and later, symmetry was used to facilitate the search for optimal codes~\cite{IversonJM:16,IversonJM:17,IversonM:18,Kopp:18,BodmannK:18,King:19,IversonM:19}.
In many cases, the symmetries that underly optimal codes can be abstracted to weaker combinatorial structures that produce additional codes.
For example, one may use strongly regular graphs to obtain optimal codes in $\operatorname{Gr}(1,\mathbf{R}^d)$~\cite{Waldron:09}, or use Steiner systems to obtain optimal codes in $\operatorname{Gr}(1,\mathbf{C}^d)$~\cite{FickusMT:12}.
In this spirit, several infinite families of optimal codes have been constructed from combinatorial designs~\cite{JasperMF:14,King:16,FickusMJ:16,FickusJMP:18,FickusJMPW:19,FickusJM:17,FickusJ:19,FickusS:20,FickusIJK:20}.
In some cases, it is even possible to construct optimal codes from smaller codes~\cite{BodmannH:16,Waldron:17,BodmannH:18,King:19b}.
Researchers have also leveraged different computational techniques to find new arrangements~\cite{ApplebyCFW:18,HughesW:18,JasperKM:19} and to study various properties of known arrangements~\cite{CohnW:12,FickusJM:18,FickusJKM:18,MixonP:19,MixonP:19b}.
See~\cite{BannaiB:09,FickusM:15,Waldron:18} for surveys of many of these results.

To date, the vast majority of this work has focused on the special case of projective spaces, and it is easy to explain this trend: it is harder to interact with points in general Grassmannian spaces.
To illustrate this, suppose you are given two tuples $A$ and $B$ of $r$-dimensional subspaces of $\mathbf{R}^d$.
You are told that the subspaces in $A$ were drawn independently and uniformly at random from the Grassmannian $\operatorname{Gr}(r,\mathbf{R}^d)$, and that $B$ was drawn according to one of two processes:
either there exists an orthogonal transformation $g\in\operatorname{O}(d)$ such that $B=g\cdot A$, or $B$ was also drawn independently and uniformly at random.
How can you tell which process was used to construct $B$?
In the special case where $r=1$, the lines are almost surely not orthogonal, and one may leverage this feature to select vector representatives of the lines and then compute a canonical form of the Gramian (i.e., the reduced signature matrix discussed in~\cite{Waldron:18}) that detects whether there exists $g\in\operatorname{O}(d)$ such that $B=g\cdot A$.
However, if $r>1$, it is not obvious how to find such an invariant. On the other hand, we benefit from the fact that many optimal configurations -- like so-called \emph{equiangular tight frames} -- share certain properties -- like not being pairwise orthogonal  -- with generic configurations, allowing results about generic configurations to be applied to certain optimal configurations.

This obstruction has had substantial ramifications on progress toward optimal codes in more general Grassmannian spaces.
In particular, Sloane maintains an online catalog~\cite{Sloane:online} of putatively optimal codes in $\operatorname{Gr}(r,\mathbf{R}^d)$ for $r\in\{1,2,3\}$ and $d\in\{3,\ldots,16\}$.
Suppose one were to find a code for $r\in\{2,3\}$ that is competitive with Sloane's corresponding putatively optimal code.
Are these codes actually the same up to rotation?
If researchers cannot easily answer this question, then they are less inclined to contribute to the hunt for optimal codes in these more general Grassmannian spaces.
\ejk{Alternatively, optimal packings that are different (even up to unitary transformations) might have different structures which may be exploited. For example, different structures that appear in distinct optimal configurations of $9$ points in $\operatorname{Gr}(1,\mathbf{C}^3)$ \cite{H:07} led to progress on Zauner's conjecture in quantum information theory \cite{ApplebyBDF:17,DangBBA:13}, along with infinite classes of optimal packings with nice matroidal structures \cite{BodmannK:18}.}
While there are several works in the literature that treat related problems~\cite{Specht:40,Wiegmann:62,Pearcy:62,KaluzninH:66,Shapiro:91,GrochowQ:23,GrochowQ:24}, the particular problem we identify has yet to be treated.
The primary purpose of this paper is to help close this gap with both theory and code.

Notationally, we let $\mathbf{F}$ denote an arbitrary field.
Every $\Gamma\leq \operatorname{GL}(d,\mathbf{F})$ has a natural action on $\operatorname{Gr}(r,\mathbf{F}^d)$.
Given an involutive automorphism $\sigma$ of $\mathbf{F}$, we consider the Hermitian form defined by $\langle x,y\rangle=\sum_i \sigma(x_i)y_i$, and we let $\operatorname{U}(d,\mathbf{F},\sigma)$ denote the subgroup of all $g\in\operatorname{GL}(d,\mathbf{F})$ such that $\langle gx,gy\rangle=\langle x,y\rangle$ for all $x,y\in\mathbf{F}^d$.
For example, $\operatorname{U}(d,\mathbf{F},\sigma)$ contains all $d\times d$ permutation matrices.  
Over any field, the identity is an involutive automorphism and over quadratic extensions one may choose the only nontrivial field automorphism as the involution (see, e.g., \cite{GreavesIJM20a,GreavesIJM20b}).  Specifically, we also adopt the standard notations for the orthogonal group $\operatorname{O}(d)=\operatorname{U}(d,\mathbf{R},\operatorname{id})$ and the unitary group $\operatorname{U}(d)=\operatorname{U}(d,\mathbf{C},\overline{\phantom{ }\cdot\phantom{ }})$.
For $\mathbf{F}\in\{\mathbf{R},\mathbf{C}\}$, we say that \textbf{generic} points in $\operatorname{Gr}(r,\mathbf{F}^d)$ satisfy property $P$ if there exists an open and dense subset $S\subseteq\mathbf{F}^{d\times r}$ such that for every $A\in S$, it holds that $V:=\operatorname{im}A\in\operatorname{Gr}(r,\mathbf{F}^d)$ and $V$ has property $P$.
\ejk{Throughout this paper, this open and dense subset turns out to be the complement of the zero set of a nonzero real polynomial, and so one may think of genericity in terms of the associated Zariski topology.}

Our problem can be viewed as an instance of a more general, fundamental problem:

\begin{problem}[Common orbit]
Given a $G$-set $X$ and two points $x,y\in X$, determine whether there exists $g\in G$ such that $g\cdot x=y$.
\end{problem}

One attractive approach to solving the common orbit problem is to construct an \textbf{invariant}, that is, a function $f\colon X\to S$ for some set $S$ such that $f(x)=f(y)$ only if there exists $g\in G$ such that $g\cdot x=y$.
In particular, $f(x)$ is determined by the orbit $G\cdot x$.
If $f$ always returns different values for different orbits, then we say $f$ is a \textbf{\ejk{complete} invariant}.
Observe that a \emph{complete} invariant provides a complete solution to common orbit \ejk{(hence the name)}, since one may simply compare $f(x)$ with $f(y)$.

We will study \ejk{three} types of common orbit problems with $X=(\operatorname{Gr}(r,\mathbf{F}^d))^n$.
In particular, for $\Gamma\in\{\operatorname{U}(d,\mathbf{F},\sigma),\operatorname{GL}(d,\mathbf{F})\}$, we consider the following actions on $X$:
\[
G
=\Gamma\times S_n,
~~
(g,\pi)\cdot \lb x_i\rb_{i\in[n]}
= \lb g\cdot x_{\pi^{-1}(i)}\rb_{i\in[n]};
\qquad
G
=\Gamma,
~~
g\cdot \lb x_i\rb_{i\in[n]}
=\lb g \cdot x_i\rb_{i\in[n]}.
\]
Here, $S_n$ denotes the symmetric group on $n$ letters.
In the following section, we show that any solution to common orbit in the case of $G=\Gamma\times S_n$ can be used to solve graph isomorphism, thereby suggesting that this case is computationally hard.
Next, Section~3 treats the case $G=\Gamma\in\{O(d),U(d)\}$.
First, we show how to obtain a canonical choice of Gramian for generic real planes (i.e., points in $\operatorname{Gr}(2,\mathbf{R}^d)$), before finding injective invariants using ideas from the representation theory of $H^*$-algebras. \ejk{We note that the remaining orbit problem with $G=\Gamma=\operatorname{GL}(d,\mathbf{F})$ has been completely solved, not only for $\mathbf{F}\in\{ \mathbf{R},\mathbf{C}\}$ but also for other rings and fields~\cite{BrooksbankL:08,ChristovIK:97,Teodorescu:15,Sergeichuk:00,QiaoS:24}.} Matlab implementations of Algorithm~\ref{alg.rnop}, Algorithm~\ref{alg.fgma}, and Lemma~\ref{lem.alg iso} may be downloaded from~\cite{KingSoftware}.

\section{Isomorphism up to permutation}

An important instance of common orbit is when $G = S_m\times S_n$ acts on $X=\{0,1\}^{m\times n}$ by $(g,h)\cdot x = gxh^{-1}$.
If we restrict $X$ to only include matrices for which each column has exactly two $1$s and no two columns are equal, then $X$ corresponds to the set of incidence matrices of simple graphs on $m$ vertices and $n$ edges, and the common orbit problem corresponds to graph isomorphism:

\begin{problem}[Graph isomorphism]
Given two simple graphs $G$ and $H$, determine whether $G\sim H$, that is, there exists a bijection $f\colon V(G)\to V(H)$ between the vertices that preserves the edges; i.e., for every $u,v\in V(G)$, it holds that $\{u,v\}\in E(G)$ if and only if $\{f(u),f(v)\}\in E(H)$.
\end{problem}

In general, a \textbf{decision problem} is a pair $(P,M)$ where $P$ maps problem instances to answers $P\colon Q\to\{\mathsf{yes},\mathsf{no}\}$ and $M\colon Q\to\mathbf{N}$ measures the size of the problem instance.
For example, for the graph isomorphism problem, $Q$ is the set of all $(G,H)$, where $G$ and $H$ are both simple graphs, $P(G,H)$ returns whether $G$ and $H$ are isomorphic, and if we represent $G$ and $H$ in terms of their incidence matrices, we are inclined to take $M(G,H)=|V(G)||E(G)|+|V(H)||E(H)|$.
We say a decision problem $(P,M)$ is \textbf{GI-hard} if, \ejk{given a black box that computes $P(x)$, one may solve graph isomorphism by an algorithm that uses that black box, and that, outside of that black box, takes time that is at most polynomial in the number of vertices in the input graphs $G$ and $H$.}
For example, if we restrict the input set of graph isomorphism to only consider $(G,H)$ for which $G$ and $H$ are regular graphs, then the resulting subproblem is known to be GI-hard~\cite{ZemlyachenkoKT:85}.
\ejk{Today, the fastest known graph isomorphism algorithm in the worst case has quasipolynomial runtime~\cite{Babai:16}, though faster algorithms are available in practice~\cite{McKayP:14}. Graph isomorphism is one of a few problems that are believed to be NP-intermediate, meaning it is in NP, but not in P, and not NP-complete.}

This section is concerned with two different isomorphism problems between tuples of subspaces.
In both cases, we focus our attention on a discrete set of problem instances.
Given a field $\mathbf{F}$, take $0,1\in\mathbf{F}$ and let $Q(r,d,n,\mathbf{F})$ denote the set of $(A,B)\in (\{0,1\}^{d\times r})^n\times (\{0,1\}^{d\times r})^n$ such that $\operatorname{rank}A_i=\operatorname{rank}B_i=r$ for every $i\in [n]$.
The standard representation of $(A,B)\in Q(r,d,n,\mathbf{F})$ uses $2rdn$ bits.
With this, we may define our decision problems:
\begin{itemize}
\item
${\mathscr{P}_{\operatorname{U}}}(r,\mathbf{F},\sigma) = (P,M)$, where $\mathbf{F}$ is an arbitrary field with involutive automorphism $\sigma$, $Q=\bigcup_{d,n\geq 1}Q(r,d,n,\mathbf{F})$, $P(A,B)$ returns whether there exists $(g,\pi)\in\operatorname{U}(d,\mathbf{F},\sigma)\times S_n$ such that $(g,\pi)\cdot\lb\operatorname{im}A_i\rb_{i\in[n]}=\lb\operatorname{im}B_i\rb_{i\in[n]}$, where $d=d(A,B)$ and $n=n(A,B)$, and $M(A,B)=2\cdot r\cdot d(A,B)\cdot n(A,B)$.
\item
${\mathscr{P}_{\operatorname{GL}}}(\mathbf{F}) = (P,M)$, where $\mathbf{F}$ is an arbitrary field, $Q=\bigcup_{r,d,n\geq 1}Q(r,d,n,\mathbf{F})$, $P(A,B)$ returns whether there exists $(g,\pi)\in\operatorname{GL}(d,\mathbf{F})\times S_n$ such that $(g,\pi)\cdot\lb\operatorname{im}A_i\rb_{i\in[n]}=\lb\operatorname{im}B_i\rb_{i\in[n]}$, where $d=d(A,B)$ and $n=n(A,B)$, and $M(A,B)=2\cdot r(A,B)\cdot d(A,B)\cdot n(A,B)$.
\end{itemize}
In words, ${\mathscr{P}_{\operatorname{U}}}(r,\mathbf{F},\sigma)$ concerns isomorphism up to unitary and permutation for any fixed rank $r$, whereas ${\mathscr{P}_{\operatorname{GL}}}(\mathbf{F})$ concerns isomorphism up to linear automorphism and permutation, but with the rank no longer fixed.
As we will see, both problems are hard.
\ejk{Prior work studied the special case where $r=1$. For work relating to ${\mathscr{P}_{\operatorname{U}}}(1,\mathbf{R},\sigma)$, see \cite{ColbournC:80}. The fact that $r$ is fixed for one problem and not for the other is an artifact of our proof of hardness. In particular, by fixing $r=1$ rather than all $r$, one may consider a problem denoted as $\mathscr{P}_{\operatorname{GL}}(1,\bF)$ for arbitrary fields $\bF$.  This problem is known as monomial code equivalence and is related to permutational code equivalence, which are both known to be GI-hard~\cite{BennettHW:25,Grochow:12,PetrankR:97}. This might suggest that ${\mathscr{P}_{\operatorname{GL}}}(r,\mathbf{F})$ is GI-hard for each $r$.  However, our proof technique relies on a reduction to regular graph isomorphism in which $r$ is determined by the graph degree; n.b.\ bounded degree GI is in P, while GI is not known to be~\cite{Luks:82}}.

\begin{theorem}\label{thm:GI}\
\ejk{The following problems are GI-hard:
\begin{itemize}
\item[(a)]
${\mathscr{P}_{\operatorname{U}}}(r,\mathbf{F},\sigma)$ for every $r\in\mathbf{N}$ and every field $\mathbf{F}$ with involutive automorphism $\sigma$ and
\item[(b)]
${\mathscr{P}_{\operatorname{GL}}}(\mathbf{F})$ for every field $\mathbf{F}$.
\end{itemize}}
\end{theorem}

\begin{proof}
(a)
Fix $r$, $\mathbf{F}$ and $\sigma$.
We will use a ${\mathscr{P}_{\operatorname{U}}}(r,\mathbf{F},\sigma)$ oracle to efficiently solve graph isomorphism.
Given two simple graphs $G$ and $H$, we return $\mathsf{no}$ if $V(G)$ and $V(H)$ are of different size, or if $E(G)$ and $E(H)$ are of different size.
Otherwise, put $n:=|V(G)|$, $e=|E(G)|$ and $d:=re$, and for each graph, arbitrarily label the vertices and edges with members of $[n]$ and $[e]$, respectively.
We use this labeling of $G$ to determine $A$.
Specifically, for each $j\in[n]$, define $A_j\in\mathbf{F}^{d\times r}$ to consist of $e$ blocks of size $r\times r$, where for each $i\in[e]$, the $i$th block of $A_j$ is $I_r$ if $j$ is a vertex in edge $i$, and otherwise the block is zero.
Define $B$ similarly in terms of our labeling of $H$.
Given $(A,B)$, the ${\mathscr{P}_{\operatorname{U}}}(r,\mathbf{F},\sigma)$ oracle returns whether there exists $(g,\pi)\in\operatorname{U}(d,\mathbf{F},\sigma)\times S_n$ such that $(g,\pi)\cdot\lb\operatorname{im}A_i\rb_{i\in[n]}=\lb\operatorname{im}B_i\rb_{i\in[n]}$, and we will output this answer as our solution to graph isomorphism.
It remains to show that $G\sim H$ if and only if there exists $(g,\pi)\in\operatorname{U}(d,\mathbf{F},\sigma)\times S_n$ such that $(g,\pi)\cdot\lb\operatorname{im}A_i\rb_{i\in[n]}=\lb\operatorname{im}B_i\rb_{i\in[n]}$.
For ($\Rightarrow$), observe that a graph isomorphism determines a choice of $\pi\in S_n$ as well as a permutation of edges.
This permutation of edges can be implemented as a block permutation matrix $g\in\operatorname{U}(d,\mathbf{F},\sigma)$ so that $gA_{\pi^{-1}(i)}=B_i$, which then implies $(g,\pi)\cdot\lb\operatorname{im}A_i\rb_{i\in[n]}=\lb\operatorname{im}B_i\rb_{i\in[n]}$.
For ($\Leftarrow$), we first define two additional graphs $G'$ and $H'$, both on vertex set $[n]$.
For $G'$, say $i\leftrightarrow j$ if there exists $x\in\operatorname{im}A_i$ and $y\in\operatorname{im}A_j$ such that $\langle x,y\rangle\neq 0$.
Define $H'$ similarly in terms of $B$.
By our construction of $A$ and $B$, it holds that $G'\sim G$ and $H'\sim H$.
Furthermore, the existence of $(g,\pi)\in\operatorname{U}(d,\mathbf{F},\sigma)\times S_n$ such that $(g,\pi)\cdot\lb\operatorname{im}A_i\rb_{i\in[n]}=\lb\operatorname{im}B_i\rb_{i\in[n]}$ implies that $G'\sim H'$, meaning $G\sim H$, as desired.

(b)
Fix $\mathbf{F}$.
We will use a ${\mathscr{P}_{\operatorname{GL}}}(\mathbf{F})$ oracle to efficiently solve graph isomorphism for regular graphs, which suffices by~\cite{ZemlyachenkoKT:85}.
Without loss of generality, we may put $n:=|V(G)|=|V(H)|$, $d:=|E(G)|=|E(H)|$, and let $r$ denote the common degree of $G$ and $H$.
For each graph, arbitrarily label the vertices and edges with members of $[n]$ and $[d]$, respectively, and let $\lb e_i\rb_{i\in[d]}$ denote the identity basis in $\mathbf{F}^d$.
For each $j\in[n]$, select $A_j\in\mathbf{F}^{d\times r}$ so that its column vectors are the $r$ members of $\lb e_i\rb_{i\in[d]}$ that correspond to edges $i$ incident to vertex $j$.
Define $B$ similarly in terms of our labeling of $H$.
Given $(A,B)$, the ${\mathscr{P}_{\operatorname{GL}}}(\mathbf{F})$ oracle returns whether there exists $(g,\pi)\in\operatorname{GL}(d,\mathbf{F})\times S_n$ such that $(g,\pi)\cdot\lb\operatorname{im}A_i\rb_{i\in[n]}=\lb\operatorname{im}B_i\rb_{i\in[n]}$, and we will output this answer as our solution to graph isomorphism.
It remains to show that $G\sim H$ if and only if there exists $(g,\pi)\in\operatorname{GL}(d,\mathbf{F})\times S_n$ such that $(g,\pi)\cdot\lb\operatorname{im}A_i\rb_{i\in[n]}=\lb\operatorname{im}B_i\rb_{i\in[n]}$.
For ($\Rightarrow$), the isomorphism determines a permutation matrix $g\in\operatorname{GL}(d,\mathbf{F})$ and a permutation $\pi\in S_n$ such that such that $(g,\pi)\cdot\lb\operatorname{im}A_i\rb_{i\in[n]}=\lb\operatorname{im}B_i\rb_{i\in[n]}$.
For ($\Leftarrow$), we first define two additional graphs $G'$ and $H'$, both on vertex set $[n]$.
For $G'$, say $i\leftrightarrow j$ if $\operatorname{im}A_i\cap\operatorname{im}A_j\neq\{0\}$, and define $H'$ similarly in terms of $B$.
By our construction of $A$ and $B$, it holds that $G'\sim G$ and $H'\sim H$.
Furthermore, the existence of $(g,\pi)\in\operatorname{GL}(d,\mathbf{F})\times S_n$ such that $(g,\pi)\cdot\lb\operatorname{im}A_i\rb_{i\in[n]}=\lb\operatorname{im}B_i\rb_{i\in[n]}$ implies that $G'\sim H'$, meaning $G\sim H$, as desired.
\end{proof}

Of course, Theorem~\ref{thm:GI} does not mean that solving ${\mathscr{P}_{\operatorname{U}}}(r,\mathbf{F},\sigma)$ or ${\mathscr{P}_{\operatorname{GL}}}(\mathbf{F})$ is always hopeless. \ejk{(In particular, graph isomorphism is solvable in practice~\cite{McKayP:14}.)} As an example, the Bargmann invariants computed in Section~\ref{subsec:lines} are ordered lists of numbers; if the histograms of these numbers are not equal, then the lines cannot be isomorphic up to permutation.

\section{Isomorphism up to linear isometry}

While the previous section demonstrated that certain isomorphism problems are hard, this section will show that isomorphism up to linear isometry is relatively easy.
This would have taken Halmos by surprise, as he considered this problem to be difficult even for triples of subspaces~\cite{Halmos:70}.
Throughout this section, we assume $\mathbf{F}\in\{\mathbf{R},\mathbf{C}\}$ without mention, meaning $\operatorname{U}(d,\mathbf{F},\sigma)\in\{\operatorname{O}(d),\operatorname{U}(d)\}$.

\subsection{Lines}\label{subsec:lines}

Chien and Waldron~\cite{ChienW:16} provide \ejk{a complete} invariant for tuples of lines in $\mathbf{F}^d$ up to isometric isomorphism.
Given a tuple $\lb v_i\rb_{i\in[n]}$ of unit vectors in $\mathbf{F}^d$ that span each line in the tuple $\mathscr{L}=\lb\ell_i\rb_{i\in[n]}$, define\footnote{Our definition differs slightly from~\cite{ChienW:16} since our inner product is conjugate-linear in the first argument.} the \textbf{$m$-vertex Bargmann invariants} or \textbf{$m$-products} by
\[
\Delta(v_{i_1},\ldots,v_{i_m})
:=\langle v_{i_1},v_{i_2}\rangle\langle v_{i_2},v_{i_3}\rangle\cdots\langle v_{i_m},v_{i_1}\rangle,
\qquad
i_1,\ldots,i_m\in[n].
\]
Denoting $P_i:=v_iv_i^*$, we see that $\Delta(v_{i_1},\ldots,v_{i_m})=\operatorname{tr}(P_{i_1}\cdots P_{i_m})$, and so the choice of $v_i\in\ell_i$ is irrelevant.
Furthermore, as their name suggests, these quantities are invariant to isometric isomorphism, since for $Q\in\operatorname{U}(d,\mathbf{F},\sigma)$, the orthogonal projection onto $Q\cdot\ell_i$ is $QP_iQ^*$, and $\operatorname{tr}(QP_{i_1}Q^*\cdots QP_{i_m}Q^*)=\operatorname{tr}(P_{i_1}\cdots P_{i_m})$.

\ejk{Let's take a moment to discuss the relationship to classical invariant theory.
In the special case where $\mathbf{F}=\mathbf{R}$, we are interested in the orbit of $\lb v_i\rb_{i\in[n]}\in(\mathbf{R}^d)^n$ under the action of $\operatorname{O}(d)\times\operatorname{O}(1)^n$.
Any polynomial that is invariant to this group is invariant to the subgroup $\operatorname{O}(d)$, and is therefore a polynomial of $\lb x_{ij}:=\langle v_i,v_j\rangle\rb_{1\leq i\leq j\leq n}$ by the first fundamental theorem of invariant theory for the orthogonal group.
Next, if we apply the Reynolds operator of $\operatorname{O}(1)^n$ to any monic monomial of the $x_{ij}$'s, the result is either zero or the same monomial, with the later case occurring precisely when the multiset of indices $ij$ that appear in the monomial form the edges of a (not necessarily simple) graph with vertex set $[n]$ in which every vertex has even degree.
Since every such graph can be decomposed into cycles, it follows that the $m$-products with $m\in[n]$ together generate the algebra of polynomial invariants, which in turn separates the orbits.
As we discuss below, Chien and Waldron~\cite{ChienW:16} identify a much smaller subset of $m$-products separate these orbits.}

Given the 2-products, one may define the \textbf{frame graph} $G(\mathscr{L})$ on $[n]$ in which we draw an edge $i\leftrightarrow j$ when $\ell_i$ and $\ell_j$ are not orthogonal; we note that the frame graph has also been referred to as the \textit{correlation network}~\cite{Strawn:07}.
Letting $E$ denote the edge set of the frame graph, then the indicator functions of the edge sets of Eulerian subgraphs of $G(\mathscr{L})$ form a subspace $\mathscr{E}\subseteq \mathbf{F}_2^E$.
Given a maximal spanning forest $F$ of $G(\mathscr{L})$, then each edge in $E\setminus E(F)$ completes a unique cycle with this forest, and the indicator functions of the edge sets of these cycles form a basis for $\mathscr{E}$.
Let $C(F)$ denote the set of these cycles.
With these notions, we may enunciate the main result of~\cite{ChienW:16} (for unweighted lines).

\begin{proposition}[Corollary~3.2 in~\cite{ChienW:16}, cf.\ Theorem~2 in~\cite{GallagherP:77}]
\label{prop.chien-waldron}
Given a tuple $\mathscr{L}$ of lines in $\mathbf{F}^d$, select any maximal spanning forest $F$ of the frame graph $G(\mathscr{L})$.
Then $\mathscr{L}$ is determined up to isometric isomorphism by its $2$-products and each $m$-product corresponding to a cycle in $C(F)$.
\end{proposition}
\begin{proof}
Let $\cL$ and $\cL'$ be $n$-tuples of lines in $\mathbf{F}^d$.  As noted above, if $\cL$ and $\cL'$ are isometrically isomorphic, then all of their $m$-products must be equal.

For the other direction, select an $n$-tuple of unit vectors $\{v_i\}_{i=1}^n$ in $\bF^d$ that span the lines in $\cL$.  Let $\{u_i\}_{i=1}^n$ be another $n$-tuple of unit vectors in $\bF^d$ such that the $2$-products and each $m$-product corresponding to a cycle in $C(F)$ (corresponding to a spanning forest $F$ of $\cL$) of each tuple of vectors are equal. We would like to show that $\{v_i\}_{i=1}^n$ and $\{u_i\}_{i=1}^n$ are the same modulo $\operatorname{U}(d,\mathbf{F},\sigma)$ and choice of basis vectors. Since the spectral theorem implies tuples of vectors are the same modulo $\operatorname{U}(d,\mathbf{F},\sigma)$ if and only if their Gramians are component-wise equal, it suffices to show that there exist unimodular $\eta_i$ for $i\in [n]$ such that for all $i,j \in [n]$
\beq\label{eqn:baseIP}
\ip{u_i}{u_j} = \overline{\eta_i}\eta_j \ip{v_i}{v_j}. 
\eeq
If $i$ and $j$ are in different components of $G(\mathscr{L})$ (where we are using $i$ as shorthand for $\ell_i$),  then $\ip{v_i}{v_j}=0$ and~\eqref{eqn:baseIP} yields no restriction on the values of $\eta_i$ and $\eta_j$.  Thus, we may assume without loss of generality that  $G(\mathscr{L})$ is connected and $F$ is a spanning tree with root $r$. Since $2$-products are equal,
\[
\absip{u_i}{u_j}^2 = \ip{u_i}{u_j}\ip{u_j}{u_i}= \ip{v_i}{v_j}\ip{v_j}{v_i}=\absip{v_i}{v_j}^2
\]
for all $i,\, j \in [n]$.  For $i \in [n]$ such that $ri$ is an edge in $F$, let $\eta_i$ be the necessarily unimodular scalar such that $\ip{u_r}{u_i}=\eta_i \ip{v_r}{v_i}$. Now for $j \in [n]$ such that $ri$ and $ij$ are edges in $F$ but not $rj$ let $\eta_j$ be the necessarily unimodular scalar such that~\eqref{eqn:baseIP} holds.  Continue this process inductively, setting the $\eta_k$ for vertices $k$ at distance $3, 4, \hdots$ from $r$. Since $F$ is spanning, we have uniquely defined $\eta_i$ for each $i \in [n]$. However, we now need to verify that~\eqref{eqn:baseIP} holds for any $ij$ that is an edge in $G(\mathscr{L})$ but not $F$.  Let $ij$ be such an edge; it lies in a unique cycle in $C(F)$, say with vertex sequence $i, j, k_3, k_4, \hdots, k_m, i$. Since each edge but $ij$ lies in $F$,
\begin{align*}
\ip{v_i}{v_j}\ip{v_j}{v_{k_3}}\ip{v_{k_3}}{v_{k_4}} \cdots \ip{v_{k_m}}{v_{i}} &=
\ip{u_i}{u_j}\ip{u_j}{u_{k_3}}\ip{u_{k_3}}{u_{k_4}} \cdots \ip{u_{k_m}}{u_{i}} \\
&=\ip{u_i}{u_j}  \overline{\eta_j}\eta_{k_3}  \ip{v_j}{v_{k_3}} \overline{\eta_{k_3}}\eta_{k_4}\ip{v_{k_3}}{v_{k_4}} \cdots \overline{\eta_{k_m}}\eta_i \ip{v_{k_m}}{v_{i}} \\
&= \eta_i  \overline{\eta_j}\ip{u_i}{u_j} \ip{v_j}{v_{k_3}}\ip{v_{k_3}}{v_{k_4}} \cdots \ip{v_{k_m}}{v_{i}},
\end{align*}
implying that~\eqref{eqn:baseIP} holds for $ij$, as desired.
\end{proof}
Generically (or for equiangular tight frames and certain other optimal configurations), none of the inner products $\langle v_i,v_j\rangle$ equal zero.
In this case, the frame graph is complete, and so we may take $F$ to be the star graph in which $1\leftrightarrow j$ for every $j\neq 1$.
Then $C(F)$ consists of all triangles in $K_n$ that have $1$ as a vertex.
Alternatively, we can put the Gramian $A=\lb\langle v_i,v_j\rangle\rb_{i,j\in[n]}$ in a canonical form by taking $D=\operatorname{diag}(\operatorname{sgn}\langle v_1,v_1\rangle,\ldots,\operatorname{sgn}\langle v_1,v_n\rangle)$ and $G=DAD^*$.
Here, $\operatorname{sgn}(re^{i\theta})=e^{i\theta}$, and so $DAD^*$ has all positive entries in its first row and column.
We refer to $G$ as the \textbf{normalized Gramian} of $A$.
Since the Gramian of $\lb v_i\rb_{i\in[n]}$ is invariant to isometries acting on $\lb v_i\rb_{i\in[n]}$, normalizing the Gramian removes any ambiguity introduced by selecting $v_i\in\ell_i$, and so the normalized Gramian is a generically injective invariant for $(\operatorname{Gr}(1,\mathbf{F}^d))^n$ modulo $\operatorname{U}(d,\mathbf{F},\sigma)$.
Notice that the entries of $G$ are the triple products corresponding to $C(F)$, and so this conclusion may also be viewed in terms of Proposition~\ref{prop.chien-waldron}.

At this point, we can treat the case of lines from two related but different perspectives:
Generically (and for certain optimal configurations), it suffices to compute the normalized Gramian, but in general, we must appeal to more intricate Bargmann invariants.
In what follows, we will see that a similar story holds for general subspaces.

\subsection{Real, nowhere orthogonal planes}

We say two subspaces $U,V\subseteq\mathbf{F}^d$ are \textbf{nowhere orthogonal} if $U\cap V^\perp=U^\perp\cap V=\{0\}$.
By counting dimensions, one may conclude that subspaces are nowhere orthogonal only if they have the same dimension.
Given bases $\lb u_i\rb_{i\in[r]}$ and $\lb v_i\rb_{i\in[r]}$ for $U$ and $V$, nowhere orthogonality is equivalent to the cross Gramian $\lb\langle u_i,v_j\rangle\rb_{i,j\in[r]}$ being invertible.
As one might expect, nowhere orthogonality is a generic property of subspaces of common dimension; we provide a short proof in the real case:

\begin{lemma}
Two generic $r$-dimensional subspaces of $\mathbf{R}^d$ are nowhere orthogonal.
\end{lemma}

\begin{proof}
Given $A_1,A_2\in\mathbf{R}^{d\times r}$, then $\operatorname{im}A_1$ and $\operatorname{im}A_2$ are nowhere orthogonal subspaces of dimension $r$ if and only if $f(A_1,A_2):=\det(A_1^* A_2)\neq0$.
Since the polynomial $f$ is nonzero at $A_1=A_2=[I_r;0]$, it follows that $f\neq0$, and so $f^{-1}(\mathbf{R}\setminus\{0\})$ is a generic set, as desired.
\end{proof}

In this section, we consider the special case of nowhere orthogonal $2$-dimensional subspaces of $\mathbf{R}^d$.
This case is particularly relevant to the study of real equi-isoclinic planes, which have received some attention recently~\cite{Et-Taoui:00,Et-Taoui:06,Et-Taoui:07,Et-Taoui:18,King:19b}.
In general, subspaces are said to be \textbf{equi-isoclinic} if there exists $\theta>0$ such that every principal angle between any two of the subspaces equals $\theta$.
(Note that equi-isoclinic subspaces with $\theta<\frac{\pi}{2}$ are nowhere orthogonal.)
Such subspaces were introduced by Lemmens and Seidel~\cite{LemmensS:73}, and at times, they emerge as arrangements of points in the Grassmannian that maximize the minimum chordal distance~\cite{DhillonHST:08}.
In fact, most of Sloane's chordal-distance codes of real planes~\cite{Sloane:online} are nowhere orthogonal, and well over half have the property that all cross Gramians have a minimum singular value greater than $10^{-4}$.

In what follows, we obtain a normalized Gramian for real, nowhere orthogonal planes, and to do so, we exploit several features of this special case.
For example, the singular values of a cross Gramian between two planes are either all equal or all distinct.
We will also leverage consequences of the fact that $\operatorname{SO}(2)$ is abelian:

\begin{lemma}
\label{lem.so moves}
If $A\in\operatorname{SO}(2)$ and $B\in\operatorname{O}(2)$, then $A^{-1}=[\begin{smallmatrix}1 & 0\\ 0 & -1\end{smallmatrix}]~A~[\begin{smallmatrix}1 & 0\\ 0 & -1\end{smallmatrix}]$ and $AB=BA^{\det B}$.
\end{lemma}

\begin{proof}
The first claim follows from the fact that $[\begin{smallmatrix} c & -s\\ s & c\end{smallmatrix}]^{-1}=[\begin{smallmatrix} c & s\\ -s & c\end{smallmatrix}]$ when $c^2+s^2=1$.
For the second claim, if $\det B=1$, then since $\operatorname{SO}(2)$ is abelian, we have $AB=BA$.
If $\det B=-1$, then put $R=[\begin{smallmatrix}1 & 0\\ 0 & -1\end{smallmatrix}]$ and $C=BR$.
Then $C\in\operatorname{SO}(2)$, and so the first claim gives
\[
AB
=ABRR
=ACR
=CAR
=CRRAR
=CRA^{-1}
=BA^{-1}.
\qedhere
\]
\end{proof}

\begin{algorithm}[t]
\SetAlgoLined
\KwData{Gramian $A\in(\mathbf{R}^{2\times2})^{n\times n}$ of orthobases of $n$ nowhere orthogonal planes in $\mathbf{R}^d$}
\KwResult{Gramian $G\in(\mathbf{R}^{2\times2})^{n\times n}$ of another choice of orthobases}
\medskip

Put $R=[\begin{smallmatrix}1 & 0\\ 0 & -1\end{smallmatrix}]$ and $S=\operatorname{diag}(R,\ldots,R)$\\
\eIf{there exists $(k,l)$ such that $A_{kl}$ has distinct singular values}{
Let $(k,l)$ be the first such indices, lexicographically\\
Compute the singular value decomposition $A_{kl}=W_{k}\Sigma V^*$ and put $\widetilde{W}_k=W_kR$\\
For $j\neq k$, compute polar decompositions $W_k^*A_{kj}=P_jW_j^*$ and $\widetilde{W}_k^*A_{kj}=\widetilde{P}_j\widetilde{W}_j^*$\\
Put $D=\operatorname{diag}(W_1,\ldots,W_n)$ and $\widetilde{D}=\operatorname{diag}(\widetilde{W}_1,\ldots,\widetilde{W}_n)$\\
Put $G=\min(D^*AD,\widetilde{D}^*A\widetilde{D})$, lexicographically
}{
For each $(i,j)$, find $\alpha_{ij}>0$ such that $H_{ij}:=\alpha_{ij}A_{ij}\in\operatorname{O}(2)$\\
Put $H=\lb H_{ij}\rb_{i,j\in[n]}$ and $D=\operatorname{diag}(H_{11},\ldots,H_{1n})$\\
\eIf{there exists $(k,l)$ such that $\det(DHD^*)_{kl}=-1$}{
Let $(k,l)$ be the first such indices, lexicographically\\
Put $Q=((DHD^*)_{kl}R)^{-1/2}$ \hfill \textit{// either square root may be selected}\\
Put $E=\operatorname{diag}(QH_{11},\ldots,QH_{1n})$\\
Put $G=\min(EAE^*,SEAE^*S)$, lexicographically\\
}{
Put $G=\min(DAD^*,SDAD^*S)$, lexicographically\\
}
}
\caption{Canonical Gramian between real, nowhere orthogonal planes
 \label{alg.rnop}}
\end{algorithm}

\begin{theorem}
The function implemented by Algorithm~\ref{alg.rnop} is \ejk{a complete} invariant for nowhere orthogonal tuples in $(\operatorname{Gr}(2,\mathbf{R}^d))^n$ modulo $\operatorname{O}(d)$.
\end{theorem}
\ejk{We note that the Algorithm~\ref{alg.rnop} takes as input (and gives as output) a Gramian of $n$ orthonormal bases of nowhere orthogonal planes in $\mathbf{R}^d)$. That is, given an element $(\operatorname{Gr}(2,\mathbf{R}^d))^n$, one fixes an orthonormal basis (i.e., columns of a $d \times 2$ matrix $A_i$) for each of the points in $\operatorname{Gr}(2,\mathbf{R}^d)$ and then computes the Gramian of the $2n$ vectors (i.e., columns of $d \times 2n$ matrix $(A_1, A_2, \hdots A_n)$). The Gramian output by the algorithm can then be factored using spectral methods to yield a $d \times 2n$ matrix $(B_1, B_2, \hdots B_n)$ with columns orthonormal bases of points in $\operatorname{Gr}(2,\mathbf{R}^d)$.
Thus, this theorem also means that the algorithm returns the same output regardless of choice of orthonormal bases for elements of $(\operatorname{Gr}(2,\mathbf{R}^d))^n$.}
\begin{proof}
First, if two different inputs $A_1$ and $A_2$ produce the same output $G$, then by the construction of $G$ in both cases, there must exist block diagonal unitary matrices $U_1$ and $U_2$ such that $G=U_1A_1U_1^*=U_2A_2U_2^*$.
It then follows that $A_2=(U_2^*U_1)A_1(U_2^*U_1)^*$, that is, $A_1$ and $A_2$ are equivalent.
It remains to show that equivalent inputs produce identical outputs.

Take any tuple $\lb U_i\rb_{i\in[n]}$ in $\operatorname{O}(2)$ and put $U=\operatorname{diag}(U_1,\ldots,U_n)$.
We will first show that $UAU^*$ produces the same output as $A$.
Throughout, we use $\sharp$ to denote the version of calculations that come from $UAU^*$, e.g., $A^\sharp=UAU^*$.
First, $A_{kl}$ has the same singular values as $A_{kl}^\sharp=U_kA_{kl}U_l^*$, and so $(k,l)^\sharp$ exists if and only if $(k,l)$ exists for the first condition in Algorithm~\ref{alg.rnop}.
Suppose $(k,l)^\sharp=(k,l)$ does exist.
Next, there are four choices of left singular vectors of $U_kA_{kl}U_l^*$, namely, $W_k^\sharp\in\{\pm U_kW_k,\pm U_kW_kR\}$.
As such, there exists $\epsilon\in\{\pm1\}$ and $t\in\{0,1\}$ such that
\[
W_k^\sharp=\epsilon U_kW_kR^t,
\qquad
\widetilde{W}_k^\sharp=\epsilon U_kW_kR^{t+1}.
\]
Since each $A_{kj}$ is invertible, the polar decompositions are unique, and we have
\[
(W_k^*A_{kj})^\sharp
=\epsilon R^tW_k^*U_k^*U_kA_{kj}U_j^*
=\epsilon R^tW_k^*A_{kj}U_j^*
=\left\{\begin{array}{ll}
\epsilon P_jW_j^*U_j^* &\text{if }t=0\\
\epsilon \widetilde{P_j}\widetilde{W_j}^*U_j^* &\text{if }t=1.
\end{array}\right.
\]
Either way, the polar decomposition gives $W_j^\sharp=\epsilon U_jW_jR^t$.
Similarly, $\widetilde{W}_j^\sharp=\epsilon U_jW_jR^{t+1}$.
With this we see that
\[
(D^*AD)_{ij}^\sharp
=(W_i^*A_{ij}W_j)^\sharp
=R^tW_i^*U_i^*U_iA_{ij}U_j^*U_jW_jR^t
=(W_iR^t)^*A_{ij}(W_jR^t),
\]
and similarly $(\widetilde{D}^*A\widetilde{D})_{ij}^\sharp=(W_iR^{t+1})^*A_{ij}(W_jR^{t+1})$.
It follows that $\{(D^*AD)^\sharp,(\widetilde{D}^*A\widetilde{D})^\sharp\}=\{D^*AD,\widetilde{D}^*A\widetilde{D}\}$, and so $G^\sharp=G$.

Next, we suppose that no such $(k,l)$ exists.
Since $A_{ij}^\sharp=U_iA_{ij}U_j^*$ with $U_i,U_j\in\operatorname{O}(2)$, then $\alpha_{ij}^\sharp=\alpha_{ij}$ and so $H_{ij}^\sharp=U_iH_{ij}U_j^*$.
Next,
\[
(DHD^*)_{kl}^\sharp
=(H_{1k}H_{kl}H_{1l}^*)^\sharp
=U_1H_{1k}U_k^*U_kH_{kl}U_l^*U_lH_{1l}^*U_1^*
=U_1(DHD^*)_{kl}U_1^*,
\]
and so $\det(DHD^*)_{kl}^\sharp=\det(DHD^*)_{kl}$.
For the remainder of the proof, put $T=DHD^*$ and define $t$ to be $1$ if $\det U_1=-1$, and otherwise $0$.

Suppose there exists $(k,l)^\sharp=(k,l)$ such that $\det T_{kl}=-1$.
Then $Q^\sharp=\pm(U_1T_{kl}U_1^*R)^{-1/2}$, and since $Q^\sharp\in\operatorname{SO}(2)$, Lemma~\ref{lem.so moves} gives
\[
(EAE^*)_{ij}^\sharp
=\alpha_{ij}^{-1}(QT_{ij}Q^*)^\sharp
=\alpha_{ij}^{-1}Q^\sharp U_1T_{ij}U_1^*(Q^\sharp)^*
=\alpha_{ij}^{-1}U_1T_{ij}U_1^*(Q^\sharp)^{\det T_{ij}}(Q^\sharp)^*.
\]
If $\det T_{ij}=1$, then this reduces to 
\[
(EAE^*)_{ij}^\sharp
=\alpha_{ij}^{-1}U_1T_{ij}U_1^*(Q^\sharp)^{\det T_{ij}}(Q^\sharp)^*
=\alpha_{ij}^{-1}U_1T_{ij}U_1^*
=\alpha_{ij}^{-1}T_{ij}^{\det U_1}
=\alpha_{ij}^{-1}R^tT_{ij}R^t.
\]
Otherwise, $\det T_{ij}=-1$, and so
\[
(EAE^*)_{ij}^\sharp
=\alpha_{ij}^{-1}U_1T_{ij}U_1^*(Q^\sharp)^{-2}
=\alpha_{ij}^{-1}U_1T_{ij}U_1^*U_1T_{kl}U_1^*R
=\alpha_{ij}^{-1}U_1T_{ij}T_{kl}U_1^*R.
\]
Since $\det(T_{ij}T_{kl})=1$, Lemma~\ref{lem.so moves} then gives
\[
(EAE^*)_{ij}^\sharp
=\alpha_{ij}^{-1}U_1T_{ij}T_{kl}U_1^*R
=\alpha_{ij}^{-1}(T_{ij}T_{kl})^{\det U_1}R
=\alpha_{ij}^{-1}R^tT_{ij}T_{kl}R^{t+1}.
\]
Similarly, 
\[
(EAE^*)_{ij}
=\left\{\begin{array}{ll}
\alpha_{ij}^{-1}T_{ij}&\text{if }\det T_{ij}=1\\
\alpha_{ij}^{-1}T_{ij}T_{kl}R&\text{if }\det T_{ij}=-1,
\end{array}\right.
\]
meaning $\{(EAE^*)^\sharp,S(EAE^*)^\sharp S\}=\{EAE^*,SEAE^* S\}$, and so $G^\sharp=G$.

In the final case, we have $\det(DHD^*)_{ij}^\sharp=\det(DHD^*)_{ij}=1$ for every $i,j\in[n]$.
Here, Lemma~\ref{lem.so moves} gives
\[
T_{ij}^\sharp
=U_1T_{ij}U_1^*
=T_{ij}^{\det U_1}
=R^t T_{ij} R^t,
\]
meaning $\{T^\sharp,ST^\sharp S\}=\{T,STS\}$, and so $G^\sharp=G$.
\end{proof}

A Matlab implementation of Algorithm~\ref{alg.rnop} may be downloaded from~\cite{KingSoftware}.

We note that another (uglier) algorithm produces a normalized Gramian for generic rank-$r$ subspaces, but the algorithm we found does not produce a Gramian if any two of the subspaces are isoclinic (for example).
Due to this failure, we decided to not report the details of this algorithm.

\subsection{$H^*$-algebras and generalized Bargmann invariants}\label{subsec:BI}

In pursuit of \ejk{a complete} invariant for $(\operatorname{Gr}(r,\mathbf{F}^d))^n$ modulo $\operatorname{U}(d,\mathbf{F},\sigma)$, we consider traces of products of matrices, generalizing Bargmann invariants and building on the approaches in~\cite{Specht:40,Wiegmann:62,Pearcy:62,KaluzninH:66,GallagherP:77,Shapiro:91}.  There are large upper bounds on the number of traces of products that must be computed to generate \ejk{a complete} invariant on a single ($1$-tuple) $d \times d$ matrix, like $4^{d^2}$ \cite{Pearcy:62}, and we prove in Lemma~\ref{lem:lowerbd} that there is a lower bound on Bargmann invariants that must be computed in general to provide \ejk{a complete} invariant for tuples of lines. Thus, our goal of this section is to give (two different) algorithms to compute injective invariants that generalize Bargmann invariants and require a reasonable number of computations.  Neither requires genericity of the subspaces.

First, we clarify how we must use these invariants with the help of a lemma:

\begin{lemma}~\label{lem:lowerbd}
Consider any function $f\colon (\operatorname{Gr}(1,\mathbf{F}^d))^d\to\mathbf{F}^m$ such that each coordinate function of $f$ is a fixed Bargmann invariant.
Then $f$ is \ejk{a complete} invariant of $(\operatorname{Gr}(1,\mathbf{F}^d))^d$ modulo $\operatorname{U}(d,\mathbf{F},\sigma)$ only if $m\geq (d-1)!/2$.
\end{lemma}

\begin{proof}
Select $\epsilon\in\{\pm\}$ and consider the lines $\mathscr{L}_{\epsilon}$ spanned by the vectors
\[
e_{1}+e_{2},
\quad
\ldots,
\quad
e_{d-1}+e_{d},
\quad
e_{d}+\epsilon e_{1}.
\]
For both choices of $\epsilon$, the frame graph $G(\mathscr{L}_\epsilon)$ is the cycle graph $C_d$ of length $d$, and so the maximal spanning forest $F$ is a path graph.
The $2$-products of $\mathscr{L}_{+}$ equal those of $\mathscr{L}_{-}$, but the $d$-product corresponding to the lone cycle $C_d\in C(F)$ has the same sign as $\epsilon$.
As such, $\mathscr{L}_+$ is not isomorphic to $\mathscr{L}_-$ modulo $\operatorname{U}(d,\mathbf{F},\sigma)$ by Proposition~\ref{prop.chien-waldron}.
The Bargmann invariants that do not vanish on $\mathscr{L}_\epsilon$ are the ones that correspond to closed walks along $C_d$.
Of these, the Bargmann invariants that distinguish $\mathscr{L}_+$ from $\mathscr{L}_-$ are closed walks with odd winding number around $C_d$.
Overall, distinguishing $\mathscr{L}_+$ from $\mathscr{L}_-$ requires a Bargmann invariant whose closed walk is supported on all of $C_d$.

Now select $\pi\in S_d$ and $\epsilon\in\{\pm\}$ and consider the lines $\pi\cdot\mathscr{L}_{\epsilon}$ obtained by permuting the tuple $\mathscr{L}_{\epsilon}$ according to $\pi$.
Distinguishing $\pi\cdot\mathscr{L}_{+}$ from $\pi\cdot\mathscr{L}_{-}$ for every $\pi\in S_d$ requires Bargmann invariants whose closed walks are supported on each of the length-$d$ cycles in the complete graph $K_d$.
The result follows from the fact that there are $(d-1)!/2$ such cycles.
\end{proof}

Considering $(d-1)!/2$ is far too large for efficient computation, we instead accept a different type of injective invariant:
Given a tuple $\mathscr{L}$ of $n$ lines, return a collection $W$ of walks on $K_n$ as well as the Bargmann invariant of $w$ evaluated at $\mathscr{L}$ for each $w\in W$.
Note that this is the form provided by Proposition~\ref{prop.chien-waldron}, at least if $F$ were selected canonically; this can be accomplished by iteratively growing $F$ from edges in lexicographic order.

The remainder of this section considers two different generalizations of the Bargmann invariants, and we use these invariants to distinguish between tuples of subspaces modulo isometric isomorphism.
Our results for both generalizations apply ideas from the representation theory of $H^*$-algebras.
For what follows, we remind the reader that $\mathbf{F}\in\{\mathbf{R},\mathbf{C}\}$.

\begin{definition}
We say $\mathscr{A}$ is an \textbf{$H^*$-algebra} over $(\mathbf{F},\sigma)$ if
\begin{itemize}
\item[(H1)] $(\mathscr{A},+,\times,\mathbf{F})$ is a finite-dimensional associative algebra with unity,
\item[(H2)] $*\colon\mathscr{A}\to\mathscr{A}$ is a conjugate-linear involutory antiautomorphism, and
\item[(H3)] $(\cdot,\cdot)\colon\mathscr{A}\times\mathscr{A}\to\mathbf{F}$ is a Hermitian form on $\mathscr{A}$ such that
\[
(xy,z)=(y,x^*z)=(x,zy^*)
\qquad
\forall x,y,z\in\mathscr{A}.
\]
\end{itemize}
A \textbf{representation} of an $H^*$-algebra $\mathscr{A}$ is a $*$-algebra homomorphism $f\colon\mathscr{A}\to\mathbf{F}^{k\times k}$.
The corresponding \textbf{character} $\chi_f\colon\mathscr{A}\to\mathbf{F}$ is given by $\chi_f(x)=\operatorname{tr}f(x)$.
Two representations $f,g\colon\mathscr{A}\to\mathbf{F}^{k\times k}$ are \textbf{equivalent} if there exists $U\in\operatorname{U}(k,\mathbf{F})$ such that $g(x)=Uf(x)U^*$.
\end{definition}

\ejk{One example of an $H^*$-algebra over $\mathbf{F}$ is $\mathbf{F}^{k \times k}$, with conjugate-linear involutory antiautomorphism $*$ the adjoint and Hermitian form $(\cdot,\cdot)$ the Hilbert--Schmidt inner product.  The quaternions form an $H^*$-algebra over the reals (see, e.g., \cite{BalachandranS:86}), where $q^* = \overline{q}$ and $(q_1,q_2)=\operatorname{Re} q_1\overline{q_2}$. $H^*$-algebras have also more recently arisen in infinite-dimensional quantum mechanics \cite{AbramskyH:12}.}

\ejk{One might ask why the Hermitian form is not mentioned in the definition of a representation of an $H^*$-algebra, considering it is an important part of the structure. Homomorphisms of Hilbert spaces are continuous linear operators, which also do not explicitly involve the Hermitian form, just the topology induced from it. Since we are dealing with finite-dimensional objects, this reduces to any linear map, e.g., an algebra homomorphism.}

\begin{proposition}[Theorem~3 in~\cite{GallagherP:77}]
\label{prop.character theory}
Two representations of an $H^*$-algebra are equivalent if and only if their characters are equal.
\end{proposition}

Given $S\subseteq\mathbf{F}^{k\times k}$, let $\mathscr{A}(S)$ denote the smallest algebra with unity containing $S$.

\begin{lemma}
\label{lem.alg iso}
Consider tuples $\lb A_i\rb_{i\in[n]}$ and $\lb B_i\rb_{i\in[n]}$ over $\mathbf{F}^{k\times k}$ for which there exists $\pi\in S_n$ such that $A_i^*=A_{\pi(i)}$ and $B_i^*=B_{\pi(i)}$ for every $i\in[n]$.
Select words $\lb w_j(x_1,\ldots,x_n)\rb_{j\in[m]}$ in noncommuting variables $x_i$ such that the evaluation $\lb E_j:=w_j(A_1,\ldots,A_n)\rb_{j\in[m]}$ is a basis for $\mathscr{A}(\lb A_i\rb_{i\in[n]})$.
(Here, evaluating the word of length zero produces the identity matrix.)
There exists $U\in\operatorname{U}(k,\mathbf{F},\sigma)$ such that $UA_iU^*=B_i$ for every $i\in[n]$ if and only if
\begin{itemize}
\item[(i)] the evaluation $\lb F_j:=w_j(B_1,\ldots,B_n)\rb_{j\in[m]}$ is a basis for $\mathscr{A}(\lb B_i\rb_{i\in[n]})$,
\item[(ii)] $\operatorname{tr}(E_i^*E_j)=\operatorname{tr}(F_i^*F_j)$ for every $i,j\in[m]$,
\item[(iii)] $\operatorname{tr}(E_i^*E_jE_k)=\operatorname{tr}(F_i^*F_jF_k)$ for every $i,j,k\in[m]$, and
\item[(iv)] $\operatorname{tr}(E_i^*A_j)=\operatorname{tr}(F_i^*B_j)$ for every $i\in[m]$, $j\in[n]$.
\end{itemize}
\end{lemma}

\begin{proof}
($\Rightarrow$)
Suppose there exists $U\in\operatorname{U}(k,\mathbf{F},\sigma)$ such that $UA_iU^*=B_i$ for every $i\in[n]$.
Then $UE_iU^*=F_i$ for every $i\in[m]$, and (i)--(iv) follow immediately.

($\Leftarrow$)
First, the assumed existence of $\pi\in S_n$ implies that $\mathscr{A}(\lb A_i\rb_{i\in[n]})$ and $\mathscr{A}(\lb B_i\rb_{i\in[n]})$ are $H^*$-algebras.
Indeed, both algebras inherit (H3) from $\mathbf{F}^{k\times k}$ by taking $(x,y)=\operatorname{tr}(x^*y)$.
By (i), there is a unique linear $f\colon\mathscr{A}(\lb A_i\rb_{i\in[n]})\to\mathscr{A}(\lb B_i\rb_{i\in[n]})$ that maps $E_i\mapsto F_i$ for every $i\in[m]$.
Next, (ii) and the non-degeneracy of $(A,B)\mapsto\operatorname{tr}(A^*B)$ implies that for every $x\in\mathscr{A}(\lb A_i\rb_{i\in[n]})$, it holds that $f(x)$ is the unique $y\in\mathscr{A}(\lb B_i\rb_{i\in[n]})$ such that $\operatorname{tr}(E_i^*x)=\operatorname{tr}(F_i^*y)$ for every $i\in[m]$.
This combined with (iii) and (iv) then imply that $f$ maps $E_jE_k\mapsto F_jF_k$ for every $j,k\in[m]$ and $A_j\mapsto B_j$ for every $j\in[n]$.
The former implies that $f$ is \ejk{an} algebra isomorphism, since decomposing $x=\sum_i a_iE_i$ and $y=\sum_j b_jE_j$ gives $xy=\sum_{ij}a_ib_jE_iE_j$, which $f$ then maps to $\sum_{ij}a_ib_jF_iF_j=f(x)f(y)$.
Since $f\colon A_i\mapsto B_i$ for every $i\in[n]$, the assumed existence of $\pi\in S_n$ implies that $f$ is a $*$-algebra isomorphism.
Indeed, letting $Rw$ denote the reversal of the word $w$, then since $f$ is an algebra isomorphism, $f$ maps
\[
(w(A_1,\ldots,A_n))^*=(Rw)(A_1^*,\ldots,A_n^*)=(Rw)(A_{\pi(1)},\ldots,A_{\pi(n)})
\]
to $(Rw)(B_{\pi(1)},\ldots,B_{\pi(n)})=(w(B_1,\ldots,B_n))^*$, and so $f(x^*)=f(x)^*$ by linearity.
At this point, we consider two representations of $\mathscr{A}(\lb A_i\rb_{i\in[n]})$, namely, the identity map and $f$.
Since the identity matrix resides in both $\mathscr{A}(\lb A_i\rb_{i\in[n]})$ and $\mathscr{A}(\lb B_i\rb_{i\in[n]})$ by definition, (ii) together with linearity gives that the characters of these representations are equal, and so Proposition~\ref{prop.character theory} implies the existence of $U\in\operatorname{U}(k,\mathbf{F},\sigma)$ such that $f(x)=UxU^*$.
Since $f\colon A_i\mapsto B_i$ for every $i\in[n]$, we are done.
\end{proof}

In~\cite{GallagherP:77}, Proposition~\ref{prop.character theory} is used to prove (Theorem 4 in~\cite{GallagherP:77}) that calculating the traces of the evaluations of every possible word on the generating matrices of length between one and $4k^2$ (i.e., on the order of $n^{4k^2}$) is \ejk{a complete} invariant.  Our goal in what follows is to prune the list of necessary words to evaluate.

Overall, to determine a tuple of matrices in $\mathbf{F}^{k\times k}$ up to unitary equivalence, it suffices to specify a collection of words $\lb w_i\rb_{i\in[m]}$ that can be used to span the corresponding $H^*$-algebra, and then report traces of the form (ii)--(iv).
In the following, we show that a certain (obvious) choice of words, i.e., the result of Algorithm~\ref{alg.fgma}, is invariant to conjugation by unitary matrices and computable in polynomial time.

\begin{algorithm}[t]
\SetAlgoLined
\KwData{Matrices $\lb A_i\rb_{i\in[n]}$ in $\mathbf{F}^{k\times k}$}
\KwResult{Words $\lb w_j\rb_{j\in[m]}$ such that $\lb w_j(A_1,\ldots,A_n)\rb_{j\in[m]}$ is a basis for $\mathscr{A}(\lb A_i\rb_{i\in[n]})$}
\medskip

Put $w_1=1$ (the word of length zero)\\
Initialize $m_\mathsf{old}=0$ and $m_\mathsf{new}=1$\\
\While{
$m_\mathsf{new}>m_\mathsf{old}$
}{
Update $m_\mathsf{old}=m_\mathsf{new}$\\
\For{$i\in[n]$ and $j\in[m_\mathsf{old}]$
}{
\If{
$A_iw_j(A_1,\ldots,A_n)$ is linearly independent of $\lb w_l(A_1,\ldots,A_n)\rb_{l\in[m_\mathsf{new}]}$
}{
Put $w_{m_\mathsf{new}+1}=x_iw_j$ and update $m_\mathsf{new}=m_\mathsf{new}+1$
}
}
}

\caption{Canonical basis for matrix algebra from finite generating set
 \label{alg.fgma}}
\end{algorithm}

\begin{lemma}
\label{lem.canonical basis for fgma}
Given $\lb A_i\rb_{i\in[n]}$ in $\mathbf{F}^{k\times k}$, Algorithm~\ref{alg.fgma} returns words $\lb w_j\rb_{j\in[m]}$ such that the evaluation $\lb w_j(A_1,\ldots,A_n)\rb_{j\in[m]}$ is a basis for $\mathscr{A}(\lb A_i\rb_{i\in[n]})$.
Given $\lb UA_iU^*\rb_{i\in[n]}$ for some $U\in\operatorname{U}(k,\mathbf{F},\sigma)$, Algorithm~\ref{alg.fgma} returns the same words $\lb w_j\rb_{j\in[m]}$.
Algorithm~\ref{alg.fgma} terminates after at most $m\leq k^2$ iterations of the while loop, and each iteration can be implemented in a way that costs $O(mnk^4)$ operations.
\end{lemma}

\begin{proof}
First, consider the set of evaluations of all words at $\lb A_i\rb_{i\in[n]}$.
This set spans $\mathscr{A}(\lb A_i\rb_{i\in[n]})$, which is a subspace of $\mathbf{F}^{k\times k}$, and therefore has finite dimension.
It follows that there exists a basis among these evaluations.
Let $L$ denote the smallest possible length of the longest word in a basis.

For the moment, let us remove the constraint $m_\mathsf{new}>m_\mathsf{old}$ of the while loop.
We claim that after the $l$th iteration of the unconstrained while loop, $\operatorname{span}\lb w_j(A_1,\ldots,A_n)\rb_{j\in[m_\mathsf{new}]}$ contains all evaluations of words of length $l$.
By our initialization $w_1=1$, this holds for $l=0$.
Assume it holds for $l\geq0$.
Then every word of length $l+1$ has the form $x_iw$, where $w$ is a word of length $l$.
Evaluating then produces $A_iw(A_1,\ldots,A_n)$.
By the induction hypothesis, $w(A_1,\ldots,A_n)$ can be expressed as a linear combination of $\lb w_j(A_1,\ldots,A_n)\rb_{j\in[m_\mathsf{old}]}$.
Since we test all of $\lb A_iw_j(A_1,\ldots,A_n)\rb_{j\in[m_\mathsf{old}]}$ for linear independence in order to select $\lb w_j\rb_{j\in[m_\mathsf{new}]}$, it follows that $\operatorname{span}\lb w_j(A_1,\ldots,A_n)\rb_{j\in[m_\mathsf{new}]}$ contains $A_iw(A_1,\ldots,A_n)$.

Now suppose that the $l$th iteration of the unconstrained while loop resulted in $m_\mathsf{new}=m_\mathsf{old}$.
Then no new words were added to $\{w_j\}$ in the $l$th iteration.
In fact, for every $i\in[n]$ and $j\in[m_\mathsf{old}]$, it holds that $A_iw_j(A_1,\ldots,A_n)$ resides in $\operatorname{span}\lb w_l(A_1,\ldots,A_n)\rb_{l\in[m_\mathsf{old}]}$, and so no new words will also be added in any future iteration.
Considering $\lb w_l(A_1,\ldots,A_n)\rb_{l\in[m_\mathsf{new}]}$ forms a basis for $\mathscr{A}(\lb A_i\rb_{i\in[n]})$ by the end of the $L$th iteration, it follows that the original while loop with constraint $m_\mathsf{new}>m_\mathsf{old}$ terminates with a basis after $L+1$ iterations.

Now suppose we were instead given $\lb UA_iU^*\rb_{i\in[n]}$ for some $U\in\operatorname{U}(k,\mathbf{F},\sigma)$.
Since the map $x\mapsto UxU^*$ is a linear isometry over $\mathbf{F}^{k\times k}$, it follows that $UA_iU^*w_j(UA_1U^*,\ldots,UA_nU^*)=UA_iw_j(A_1,\ldots,A_n)U^*$ is linearly independent of
\[
\lb w_l(UA_1U^*,\ldots,UA_nU^*)\rb_{l\in[m_\mathsf{new}]}=\lb Uw_l(A_1,\ldots,A_n)U^*\rb_{l\in[m_\mathsf{new}]}
\]
if and only if $A_iw_j(A_1,\ldots,A_n)$ is linearly independent of $\lb w_l(A_1,\ldots,A_n)\rb_{l\in[m_\mathsf{new}]}$.
As a consequence, Algorithm~\ref{alg.fgma} returns the same words $\lb w_j\rb_{j\in[m]}$.

For the final claim, recall that the while loop terminates after $L+1$ iterations.
To estimate this number of iterations, let $m_l$ denote the dimension of $\operatorname{span}\lb w_l(A_1,\ldots,A_n)\rb_{l\in[m_\mathsf{new}]}$ after the $l$th iteration of the while loop, i.e., $m_l=m_\mathsf{new}$.
Then
\[
1=m_0<m_1<\cdots<m_L=m_{L+1}=m.
\]
It follows that $L<m$, and so the while loop terminates after at most $m\leq k^2$ iterations, as claimed.
One may implement each iteration of the while loop by first multiplying every matrix $A_i$ by every matrix $w_j(A_1,\ldots,A_n)$, costing $nm_\mathsf{old}\cdot O(k^3)=O(nk^5)$ operations, then vectorizing the matrices $\lb w_l(A_1,\ldots,A_n)\rb_{l\in[m_\mathsf{new}]}$ and the $nm_\mathsf{old}$ matrix products to form the columns of a $k^2\times (m_\mathsf{old}+nm_\mathsf{old})$ matrix, computing the row echelon form of this matrix in $O(k^4 (m_\mathsf{old}+nm_\mathsf{old}))=O(mnk^4)$ operations, and then finally using the pivot columns of the result to decide which words to add to $\{w_j\}$.
\end{proof}

Matlab implementations of Algorithm~\ref{alg.fgma} and Lemma~\ref{lem.alg iso} may be downloaded from~\cite{KingSoftware}.

While the per-iteration cost of Algorithm~\ref{alg.fgma} scales poorly with $k$, we will find that this cost can sometimes be improved dramatically.
At the moment, the main takeaway should be that Algorithm~\ref{alg.fgma} always returns the desired basis in polynomial time.

\subsubsection{Projection algebras}

Taking inspiration from~\cite{GallagherP:77}, and in light of Lemma~\ref{lem.alg iso}, there is a natural choice of invariant to determine tuples of subspaces up to isometric isomorphism.

\begin{theorem}
There exists \ejk{a complete} invariant for $(\operatorname{Gr}(r,\mathbf{F}^d))^n$ modulo $\operatorname{U}(d,\mathbf{F},\sigma)$ that, given a tuple of orthogonal projection matrices, can be computed in $O(nd^8+r^2d^9)$ operations.
\end{theorem}

\begin{proof}
Let $A_i$ denote the orthogonal projection onto the $i$th subspace, run Algorithm~\ref{alg.fgma} to determine words $\lb w_j\rb_{j\in[m]}$ that produce a basis $\lb E_j\rb_{j\in[m]}$ for the algebra $\mathscr{A}(\lb A_i\rb_{i\in[n]})$, and then compute the traces prescribed in Lemma~\ref{lem.alg iso}(ii)--(iv).
By Lemmas~\ref{lem.alg iso} and~\ref{lem.canonical basis for fgma}, the words $\lb w_j\rb_{j\in[m]}$ together with the traces (ii)--(iv) form \ejk{a complete} invariant for $(\operatorname{Gr}(r,\mathbf{F}^d))^n$ modulo $\operatorname{U}(d,\mathbf{F},\sigma)$.
Since Algorithm~\ref{alg.fgma} ensures that $E_1=I$, then the traces in (ii) are already captured by the traces in (iii).

To compute these traces, it is helpful to perform some preprocessing.
For each projection $A_i$, we find a decomposition of the form $A_i=T_iT_i^*$ with $T_i\in\mathbf{F}^{d\times r}$ in $O(rd^2)$ operations.
(It suffices to draw Gaussian vectors $\lb g_j\rb_{j\in[r]}$ in $O(rd)$ operations, then compute $\lb A_ig_j\rb_{j\in[r]}$ in $O(rd^2)$ operations, then perform Gram--Schmidt in $O(dr^2)$ operations.)
Every trace that we need to compute can be expressed as the trace of a product of $A_i$'s.
We will apply the cyclic property of the trace and compute matrix--vector products whenever possible.
For example, the trace of $A_1A_2$ is given by
\[
\operatorname{tr}(A_1A_2)
=\operatorname{tr}(T_1T_1^*T_2T_2^*)
=\operatorname{tr}(T_1^*T_2T_2^*T_1)
=\sum_{j\in[r]}e_j^*T_1^*T_2T_2^*T_1e_j,
\]
where $\lb e_j\rb_{j\in[r]}$ denotes the identity basis in $\mathbf{F}^r$.
We compute the right-hand side by first computing $T_1e_j$ in $O(rd)$ operations, then $T_2^*(T_1e_j)$ in $O(rd)$ operations, etc.
In our case, each word has length at most $m$, and so each term of the above sum can be computed in $O(rdm)$ operations.

Overall, we compute the words in $O(m^2nd^4)$ operations (by Lemma~\ref{lem.canonical basis for fgma}), then we compute $\lb T_i\rb_{i\in[n]}$ in $O(nrd^2)$ operations, and then each of the $m^3$ traces in (iii) and each of the $mn$ traces in (iv) costs $O(r^2dm)$ operations.
In total, this invariant costs $O(m^2nd^4+nrd^2+m^4r^2d+m^2nr^2d)$ operations.
Since $m\leq d^2$, this operation count is $O(nd^8+r^2d^9)$.
\end{proof}

While this invariant can be computed in polynomial time, the runtime is sensitive to the ambient dimension $d$.

\subsubsection{Quivers and cross Gramian algebras}

Consider any sequence $\lb A_i\rb_{i\in[n]}$, where each $A_i$ is an isometric embedding of some $r$-dimensional vector space $V_i$ over $\mathbf{F}$ into $\mathbf{F}^d$.
That is, $A_i\colon V_i\hookrightarrow\mathbf{F}^d$ and $\lb\operatorname{im}A_i\rb_{i\in[n]}\in(\operatorname{Gr}(r,\mathbf{F}^d))^n$.
For every $(i,j)\in[n]^2$, we then have a mapping $A_i^*A_j\colon V_j\to V_i$.
Together, $(\lb V_i\rb_{i\in[n]},\lb A_i^*A_j\rb_{i,j\in[n]})$ forms a representation of a so-called \textit{quiver} $Q=(Q_0,Q_1,s,t)$ defined by $Q_0=[n]$, $Q_1=[n]^2$, $s\colon(i,j)\mapsto j$, and $t\colon(i,j)\mapsto i$.
The corresponding quiver algebra $\mathbf{F}Q$ enjoys a representation over $V:=\bigoplus_{i\in [n]}V_i$ with maps
\[
f_{ij}\colon V \xrightarrow[\hspace{1.5cm}]{\pi_j} V_j\xrightarrow[\hspace{1.5cm}]{A_i^*A_j} V_i \xhookrightarrow[\hspace{1.5cm}]{\pi_i^*} V,
\]
where $\pi_i$ denotes the coordinate projection from $V$ to $V_i$.
As we will see, these endomorphisms over $V$ generate an $H^*$-algebra that provides more efficient invariants.

\begin{theorem}
There exists \ejk{a complete} invariant for $(\operatorname{Gr}(r,\mathbf{F}^d))^n$ modulo $\operatorname{U}(d,\mathbf{F},\sigma)$ that, given a Gramian of orthobases of subspaces, can be computed in $O(r^8n^5+r^9n^3)$ operations.
\end{theorem}

\begin{proof}
Denote the subspaces by $\lb\operatorname{im}A_i\rb_{i\in[n]}$, where each $A_i\in\mathbf{F}^{d\times r}$ has orthonormal columns, and put $A=[A_1\cdots A_n]$.
By assumption, we are given the Gramian $A^*A$.
Letting $\Pi_i$ denote the $rn\times rn$ orthogonal projection matrix onto the $i$th block of $r$ coordinates in $\mathbf{F}^{rn}$, then the matrix representation of $f_{ij}$ is $A_{ij}:=\Pi_iA^*A\Pi_j$.

Given $\lb A_{ij}\rb_{i,j\in[n]}$ and $\lb B_{ij}\rb_{i,j\in[n]}$ of this form, suppose there exists $U\in\operatorname{U}(rn,\mathbf{F},\sigma)$ such that $UA_{ij}U^*=B_{ij}$ for every $i,j\in[n]$.
Then since $A_{ii}=B_{ii}=\Pi_i$, it holds that $U$ is necessarily block diagonal.
Furthermore,
\[
UA^*AU^*
=U\Big(\sum_{ij}\Pi_iA^*A\Pi_j\Big)U^*
=\sum_{ij}UA_{ij}U^*
=\sum_{ij}B_{ij}=B^*B.
\]
As such, unitary equivalence between $\lb A_{ij}\rb_{i,j\in[n]}$ and $\lb B_{ij}\rb_{i,j\in[n]}$ implies block unitary equivalence between the orthobasis Gramians $A^*A$ and $B^*B$.
The implication also goes in the other direction:
Given a block diagonal $U\in\operatorname{U}(rn,\mathbf{F},\sigma)$ such that $UA^*AU^*=B^*B$, then
\[
UA_{ij}U^*
=U\Pi_iA^*A\Pi_jU^*
=\Pi_iUA^*AU^*\Pi_j
=\Pi_iB^*B\Pi_j
=B_{ij}.
\]
It remains to test whether there exists $U\in\operatorname{U}(rn,\mathbf{F},\sigma)$ such that $UA_{ij}U^*=B_{ij}$ for every $i,j\in[n]$, which leads us to consider Lemma~\ref{lem.alg iso}.

Note that $A_{ij}^*=A_{ji}$ and similarly for $B$, and so $\lb A_{ij}\rb_{i,j\in[n]}$ and $\lb B_{ij}\rb_{i,j\in[n]}$ satisfy the hypothesis of Lemma~\ref{lem.alg iso} with $\pi(i,j)=(j,i)$.
Since $\sum_iA_{ii}=I$, we may run a version of Algorithm~\ref{alg.fgma} that instead initializes with all words $x_{ij}$ of length $1$ whose evaluations $A_{ij}$ are nonzero; these evaluations are linearly independent since they have disjoint support.
Since all words of positive length evaluate as a matrix in $\mathbf{F}^{rn\times rn}$ that is supported on some $r\times r$ block, it follows that the resulting basis can be indexed as $\lb E_{ijk}\rb_{i,j\in[n],k\in[m_{ij}]}$, where $E_{ijk}$ is the $k$th basis element that is supported in the $(i,j)$th $r\times r$ block.
For example, it holds that $E_{ij1}=A_{ij}$ whenever $A_{ij}\neq0$.

To see how efficient this choice of invariants is, we first describe how to reduce the per-iteration cost of Algorithm~\ref{alg.fgma} to $O(r^6n^3)$.
First, we take all products between $A_{ij}$'s and evaluations of existing words.
For each word, there are at most $n$ different $A_{ij}$'s that will produce a nonzero product, and so the total number of products is at most $r^2n^3$, each costing $O(r^3)$ operations.
Next, the evaluations of existing words and the resulting products can be partitioned according to their support before testing linear independence.
For each $i,j\in[n]$, the total number of these matrices that are supported on the $(i,j)$th $r\times r$ block is at most $r^2+nr^2$ (at most $r^2$ from the existing words, and at most $nr^2$ from the resulting products), and it costs $O(r^6n)$ operations to compute the corresponding row echelon form.
We perform this for each of the $n^2$ blocks to identify new words to add.
All together, the per-iteration cost is $O(r^2n^3\cdot r^3 + n^2\cdot r^6 n)=O(r^6n^3)$.
Our bound on the total number of iterations is $r^2n^2$, meaning we obtain the desired words after $O(r^8n^5)$ operations.

Next, we point out the complexity of computing the traces (ii)--(iv).
First, since $E_{ii1}E_{ijk}=E_{ijk}$, the traces in (ii) are examples of traces in (iii).
Next, for every $(i,j)\in[n]$, we either have $E_{ij1}=A_{ij}$ or $x_{ij}$ is not one of the words in $\lb w_j\rb_{j\in[m]}$.
As such, the traces in (iv) are captured by both the words and the traces in (iii).
Of the traces in (iii), the only ones that are possibly nonzero take the form
\[
\operatorname{tr}(E_{jia}^*E_{jkb}E_{kic})
\]
for some $i,j,k\in[n]$, $a\in[m_{ji}]$, $b\in[m_{jk}]$ and $c\in[m_{ki}]$.
Since $m_{ij}\leq r^2$ for every $i,j\in[n]$, we therefore have at total of at most $n^3r^6$ traces to compute, each costing $O(r^3)$ operations.
These $O(r^9n^3)$ operations contribute to the total of $O(r^8n^5+r^9n^3)$ operations it takes to compute this invariant.
\end{proof}

Interestingly, the cross Gramian algebra introduced in the above proof can be used to obtain a new (short) proof of Proposition~\ref{prop.chien-waldron}:

\begin{proof}[Proof of Proposition~\ref{prop.chien-waldron}]
Consider $A=[a_1\cdots a_n]$, where each $a_i$ is a unit vector spanning the corresponding line in $\mathscr{L}$.
Then the cross Gramian algebra is generated by $A_{ij}=\langle a_i,a_j\rangle e_ie_j^*$.
Observe that every product of these matrices is either $0$ or a multiple of $e_ie_j^*$ for some $(i,j)\in[n]^2$.
Furthermore, $e_ie_j^*$ resides in the algebra precisely when $i$ and $j$ belong to a common component of the frame graph $G(\mathscr{L})$.

Given a maximal spanning forest $F$ of the frame graph $G(\mathscr{L})$, we select the following words in noncommuting variables $\lb x_{ij}\rb_{i,j\in[n]}$:
For each $(i,j)\in[n]^2$ such that $i$ and $j$ belong to a common component of $G(\mathscr{L})$, select the unique directed path in $F$ from $j$ to $i$ with vertices denoted by $j=i_0\to i_1\to \cdots\to i_l=i$, and then put
\[
w_{ij}:=x_{i_l,i_{l-1}}x_{i_{l-1},i_{l-2}}\cdots x_{i_2,i_1}x_{i_1,i_0}.
\]
In particular, $w_{ii}=x_{ii}$ for every $i\in[n]$.
Then the evaluation of $w_{ij}$ is a nonzero multiple of $e_ie_j^*$, and all of these evaluations together form a basis for the algebra.

Now consider the traces in Lemma~\ref{lem.alg iso}(ii)--(iv).
Every trace in (ii) and (iii) is either $0$ or some product of $2$-products.
Indeed, the trace is nonzero only if the corresponding directed paths form a closed walk along the edges of $F$, in which case each edge of $F$ is traversed as many times in one direction as it is in the other direction.
Meanwhile, a trace in (iv) is nonzero only if the corresponding directed paths form a closed walk comprised of a directed path in $F$ and a directed edge in $G(\mathscr{L})$.
If the path has length $1$, then the result is a $2$-product, and otherwise the result is an $m$-product corresponding to a cycle in $C(F)$.
\end{proof}

\section{Discussion}

This paper studied the problem of testing isomorphism between tuples of subspaces with respect to various notions of isomorphism.
Several open problems remain:
\begin{itemize}
\item
Is there a canonical choice of Gramian for equi-isoclinic subspaces of dimension $r>2$? What about the complex case?
\item
How many (generalized) Bargmann invariants are required to solve isomorphism up to linear isometry?
\item
How should one compute the symmetry group of a given tuple of subspaces?
\end{itemize}
Some of the ideas in the paper may have interesting applications elsewhere.
For example, there has been a lot of work to develop symmetric arrangements of points in the Grassmannian~\cite{ValeW:04,ChienW:11,BroomeW:13,Waldron:13,ValeW:16,ChienW:16,IversonJM:16,IversonJM:17,IversonM:18,Kopp:18,BodmannK:18,King:19,IversonM:19,FickusS:20,FickusIJK:20}.
What are the projection and cross Gramian algebras of these arrangements?
It would also be interesting to see if some of the techniques presented in this paper could be used to treat other emerging problems involving invariants to group actions, e.g.~\cite{BandeiraBKPWW:17,CahillCC:19}.

\section*{Acknowledgments}

The authors thank the anonymous referees who gave valuable feedback and also alerted us to the fact that the orbit problem with $G=\Gamma=\operatorname{GL}(d,\mathbf{F})$ had been solved.
DGM was partially supported by AFOSR FA9550-18-1-0107, NSF DMS 1829955, and the 2019 Kalman Visiting Fellowship at the University of Auckland.

\end{document}